\documentclass[american]{amsart}

\usepackage[utf8]{inputenc}
\usepackage[T1]{fontenc}
\usepackage{a4wide}

\usepackage{graphicx}
\usepackage[]{hyperref}
\usepackage{caption}
\usepackage{color}

\newcommand{\defin}[1]{\textbf{#1}}

\DeclareMathOperator{\id}{id}
\DeclareMathOperator{\Span}{span}

\newcommand{\restricted}[2]{{\left.{#1}\right|_{#2}}}
\newcommand{\aOT}{{\alpha_{\operatorname{OT}}}}
\newcommand{\DOT}{{\DD_{\operatorname{OT}}}}
\newcommand{\lcan}{{\lambda_{\operatorname{can}}}}
\newcommand{\p}{\partial}
\newcommand{\lie}[1]{{\mathcal{L}_{#1}}}
\newcommand{\abs}[1]{{\mathopen\lvert #1\mathclose\rvert}}
\newcommand{\norm}[1]{{\mathopen\lVert #1\mathclose\rVert}}

\newcommand{\Ureg}{{U_{\mathrm{reg}}}}
\newcommand{\Usymp}{{U_{\mathrm{symp}}}}

\renewcommand{\epsilon}{\varepsilon}

\newcommand{\0}{{\mathbf 0}}

\newcommand{\dD}{{\mathcal D}}
\newcommand{\DD}{{\mathbb D}}
\newcommand{\fF}{{\mathcal F}}

\newcommand{\NN}{{\mathbb N}}
\newcommand{\bfp}{{\mathbf p}}
\newcommand{\bfq}{{\mathbf q}}
\newcommand{\QQ}{{\mathbb Q}}
\newcommand{\RR}{{\mathbb R}}
\renewcommand{\SS}{{\mathbb S}}
\newcommand{\TT}{{\mathbb T}}
\newcommand{\x}{{\mathbf x}}
\newcommand{\y}{{\mathbf y}}

\newcommand{\vol}{\mathrm{vol}}

\theoremstyle{plain}

\newcounter{maintheorem}

\newtheorem{main_theorem}[maintheorem]{Theorem}

\newtheorem{propIntro}{Proposition}
\newtheorem{questionIntro}[propIntro]{Question}
\newtheorem{remark*}{Remark}

\newtheorem{theorem}{Theorem}[section]
\newtheorem{lemma}[theorem]{Lemma}
\newtheorem{corollary}[theorem]{Corollary}
\newtheorem{proposition}[theorem]{Proposition}

\theoremstyle{remark}
\newtheorem{remark}[theorem]{Remark}

\theoremstyle{definition}
\newtheorem{definition}{Definition}

\numberwithin{equation}{section}

\begin{document}
\title{An overtwisted convex hypersurface in higher dimensions}

\author[R.\ Chiang]{River Chiang}

\address[R.\ Chiang]{Department of Mathematics\\ National Cheng Kung
  University No.~1\\ Dasyue Rd.\\ Tainan City 70101\\ Taiwan}

\email{riverch@mail.ncku.edu.tw}

\author[K.\ Niederkrüger]{Klaus Niederkrüger-Eid}

\address[K.\ Niederkrüger]{
  Institut Camille Jordan\\
  Université Claude Bernard Lyon~1\\
  43 boulevard du 11 novembre 1918\\
  F-69622 Villeurbanne Cedex \\
  FRANCE}

\email{niederkruger@math.univ-lyon1.fr}

\begin{abstract}
  We show that the germ of the contact structure surrounding a certain
  kind of convex hypersurfaces is overtwisted.  We then find such
  hypersurfaces close to any plastikstufe with toric core so that
  these imply overtwistedness.  All proofs in this article are
  explicit, and we hope that the methods used here might hint at a
  deeper understanding of the size of neighborhoods in contact
  manifolds.

  In the appendix we reprove in a concise way that the Legendrian
  unknot is loose if the ambient manifold contains a large enough
  neighborhood of a $2$-dimensional overtwisted disk.  Additionally we
  prove the folklore result that the singular distribution induced on
  a hypersurface~$\Sigma$ of a contact manifold~$(M,\xi)$ determines
  the germ of the contact structure around $\Sigma$.
\end{abstract}

\maketitle

The fundamental distinction between tight and overtwisted contact
structures was discovered first by Eliashberg in dimension~$3$
\cite{Eliashberg_Overtwisted}, and then generalized to arbitrary
dimensions by Borman--Eliashberg--Murphy
\cite{BormanEliashbergMurphyExistence}.  The main feature is that an
overtwisted contact structure is flexible.  However, the original high
dimensional definition makes it practically unverifiable if a given
contact structure is overtwisted. Thanks to Casals--Murphy--Presas
\cite{CasalsMurphyPresas} we know that most previously existing
conjectural definitions for overtwistedness are actually equivalent to
the one given in \cite{BormanEliashbergMurphyExistence}.

Still many open questions persist.  One of them was lifted by Huang
\cite{HuangPlastistufe}: The second author of this article proposed a
definition of overtwistedness called plastikstufe
\cite{NiederkrugerPlastikstufe}.  The results in
\cite{MurphyEtAlLoosePlastikstufe} combined with those of
\cite{CasalsMurphyPresas} showed that certain very special
plastikstufes imply overtwistedness, while Huang explains in his
article that this result extends to \emph{any} plastikstufe.
Nevertheless, several points in his proof are unclear to us and more
detailed arguments may be desired.

In this article, we reprove Huang's result for the more restrictive
case of plastikstufes with toric core (generalizing results by Adachi
\cite{AdachiLoosePlastikstufe}).

\medskip

Our strategy is based on the following observation about a certain
kind of convex hypersurfaces.  Consider for any $C>0$ the manifold
\begin{equation*}
  \Sigma_C = \DD^2_{\le \pi} \times (-C,C)^{2n}
\end{equation*}
carrying a singular distribution~$\dD_C$ given as the kernel of the
$1$-form $\beta = r\, \sin r\, d\vartheta - \sum_{j=1}^n t_j\, ds_j$,
where $(r,\vartheta)$ are polar coordinates on the disk, and
$(s_j,t_j)$ are the natural coordinates on the cube
$(- C, C)^n\times (-C,C)^n$.

\begin{main_theorem}\label{thm: overtwisted convex disk}
  There exists a constant~$C_{\mathrm{OT}}>0$ such that every contact
  manifold~$(M,\xi)$ of dimension $\ge 5$ is overtwisted, if it is
  possible to embed a hypersurface $(\Sigma_C, \dD_C)$ with
  $C > C_{\mathrm{OT}}$ such that $\xi$ induces the singular
  distribution~$\dD_C$ on $\Sigma_C$.
\end{main_theorem}

Consequently, we call any embedded hypersurface $(\Sigma_C, \dD_C)$
with $C > C_{\mathrm{OT}}$ an \defin{overtwisted convex disk}.

\begin{remark*}
  \begin{itemize}
  \item [(a)] This definition is closely related to the
    characterization of overtwisted contact structures in terms of
    ``large neighborhoods'' given in
    \cite{NiederkrugerPresas_neighborhoods, CasalsMurphyPresas}.  Our
    result states that instead of considering large embedded balls, it
    is already sufficient to find a large hypersurface.
  \item [(b)] In the model used in Section~\ref{sec: squeeze in high
      dimensions}, one sees directly that there is an obvious contact
    vector field $Z$ ($=\partial_z$ in fact) that is transverse to the
    overtwisted convex disk.  It is tempting to try to characterize
    the constant~$C_{\mathrm{OT}}$ more explicitly by trying to
    recognize some sort of higher dimensional Giroux criterion for
    convex hypersurfaces.  Unfortunately, the dividing set for $Z$ is
    non-compact, and we have not succeeded in finding a more suitable
    contact vector field.
  \item [(c)] Even though we are unable to give a specific value for
    the size parameter that appears in the definition of
    overtwistedness in \cite{BormanEliashbergMurphyExistence} (and the
    equivalent formulation via large neighborhoods in
    \cite{CasalsMurphyPresas}), it follows from our argument that the
    size parameter can be chosen uniformly for all dimensions.
  \end{itemize}
\end{remark*}

\begin{proof}[Proof of Theorem~\ref{thm: overtwisted convex disk}]
  We show in Section~\ref{sec: overtwisted dim 3} and \ref{sec:
    squeeze in high dimensions} that every neighborhood of
  $\DD^2_{\le \pi} \times \bigl(-C,C\bigr)^{2n}$ contains the
  embedding of a certain type of open subset
  $B(h)\times \bigl(-\frac{5}{6}\,C, \frac{5}{6}\,C\bigr)^{2n}$.  See
  Corollary~\ref{cor: can embedded high domain close to hypersurface}
  for the details.  As explained first in
  \cite[p.~1813]{MurphyEtAlLoosePlastikstufe}, the Legendrian unknot
  is loose, if $C$ is chosen sufficiently large.  We give in
  Appendix~\ref{sec: overtwisted then loose} a streamlined proof of
  this statement.

  It was proved in \cite{CasalsMurphyPresas} that any contact manifold
  in which the unknot is loose is overtwisted.
\end{proof}

\begin{questionIntro}
  Is it possible to explicitly show that in any neighborhood of a
  hypersurface~$\Sigma_C$ with $ C > C_{\mathrm{OT}}$ one can embed a
  hypersurface~$\Sigma_{C'}$ with $C' > 2C$ as in the analogous claim
  for loose charts \cite[Proposition~4.4]{MurphyLoose}?
\end{questionIntro}

It is likely that the contact germs around the overtwisted convex
disk~$(\Sigma_C, \dD_C)$ correspond to the thick neighborhoods which
were shown in \cite{CasalsMurphyPresas} to be overtwisted.
Unfortunately the formulation of thick neighborhoods in
\cite{CasalsMurphyPresas} and the one used here are not directly
comparable, and the verification of the equivalence would force us to
dig through the proofs of the corresponding statements in
\cite{CasalsMurphyPresas}.  We have preferred to take instead the
detour over the loose unknots which allows us to split the argument
cleanly into separate steps.

\medskip

Even if it might be unclear at the moment if Theorem~\ref{thm:
  overtwisted convex disk} is more than a curious observation, it
allows us to prove in an extremely elementary way that every contact
manifold containing a plastikstufe with \emph{toric} core is
overtwisted, as claimed in \cite{HuangPlastistufe}.

\medskip

In Appendix~\ref{sec: overtwisted then loose}, we show that the
Legendrian unknot is loose in a large neighborhood of an overtwisted
$2$-disk; in Appendix~\ref{sec: contact germ by hypersurface} we prove
the folklore result that a hypersurface in a contact manifold
determines together with the induced singular distribution the germ of
the contact structure.

\subsection*{Acknowledgments}

We thank Sylvain Courte and Emmanuel Giroux and Patrick Massot for
useful and interesting discussions.  During a short conversation with
Emmy Murphy, we learned that she had come to similar conclusions
regarding the ``height'' of the overtwisted model.

\section{A plastikstufe with toric core implies overtwistedness}

Denote the cylindrical coordinates on $\RR^3$ by $(r, \vartheta, z)$.
The $1$-form
\begin{equation*}
  \aOT = \cos(r)\,dz + r\sin(r)\, d\vartheta
\end{equation*}
is then a well-defined contact form.  The disk
$\DOT := \bigl\{(r,\vartheta,z) \mid\, r\le \pi, \; z = 0 \bigr\}$ is
overtwisted, and we will call it the \emph{standard} overtwisted disk.

We choose a small cylindrical box of height~$h$ around $\DOT$ of the
form
\begin{equation}
  B(h) := \DD^2_{< \pi + \delta} \times (-h,h) \;.  \label{eq: box}
\end{equation}

\medskip

Let $(M,\xi)$ be a $(2n+3)$-dimensional contact manifold that contains
a plastikstufe of the form $\DOT \times \TT^n$.  We will show that we
can find an arbitrarily ``large''
hypersurface~$\DD^2_{\le \pi} \times [-C,C]^{2n}$ in any neighborhood
of the plastikstufe by successively unwinding each of the
$\SS^1$-factors of the torus.  This way we obtain the following
corollary.

\begin{corollary}\label{cor: toric plastikstufe overtwisted}
  Every contact manifold that contains a plastikstufe
  $\DOT \times \TT^n$ with toric core, also admits an embedding of an
  overtwisted convex disk.
\end{corollary}
\begin{proof}
  There is a neighborhood of the plastikstufe that is contactomorphic
  to an open neighborhood of $\DOT\times \TT^n$ in
  \begin{equation*}
    \bigl( \RR^3\times T^*\TT^n, \aOT + \lcan\bigr) \;;
  \end{equation*}
  compare \cite[Theorem~I.1.3.]{NiederkrugerHabilitation}.  Choosing
  $\epsilon > 0$ and $\delta > 0$ small enough, we can assume that
  this neighborhood is contactomorphic to a product of the form
  $B(\epsilon) \times \DD_{<\delta}\bigl(T^*\TT^n\bigr)$, where
  $B(\epsilon) \subset \RR^3$ is a cylindrical box as defined in
  \eqref{eq: box} and $\DD_{<\delta}\bigl(T^*\TT^n\bigr)$ is the disk
  bundle of radius $\delta$ in $T^*\TT^n$.

  We can now apply to this neighborhood Lemma~\ref{lemma: unwrap one
    circle} in dimension~$5$, or Lemma~\ref{lemma: unwrap torus} in
  the general case to find a hypersurface of the form
  $\DD^2_{\le \pi} \times (-C,C)^n\times \bigl(- a,a\bigr)^n$ for
  $a = \tfrac{\delta}{2\sqrt{n}}$, and where $C>0$ can be chosen to be
  arbitrarily large.  The singular distribution induced by the contact
  structure agrees with the kernel of
  $r\, \sin r\, d\vartheta - \sum_{j=1}^n t_j\, ds_j$, where
  $(r,\vartheta)$ are polar coordinates on the disk, and $(s_j,t_j)$
  are the natural coordinates on the rectangle
  $(- C, C)\times (-a,a)$.

  If $C>0$ is chosen larger than $2 C_{\mathrm{OT}}^2 /a$, then it
  suffices to apply to each coordinate pair~$(s_j,t_j)$ the
  diffeomorphism
  $(s_j,t_j)\mapsto \bigl(\mu^{-1}\, s_j, \mu\, t_j\bigr)$ with
  $\mu = \frac{2C_{\mathrm{OT}}}{a}$ to contain the desired
  overtwisted convex disk $(\Sigma_{\tilde{C}}, \dD_{\tilde{C}})$ for
  some appropriate $\tilde{C}> C_{\mathrm{OT}}$.
\end{proof}

We will now show how to ``unwrap'' the toric plastikstufe.  Consider
for simplicity first a contact manifold of dimension~$5$ so that
$\TT^n = \SS^1$.

\begin{lemma}\label{lemma: unwrap one circle}
  Choose any $\epsilon > 0$ and $\delta > 0$, and let
  $B(\epsilon) \times \DD_{<\delta}\bigl(T^*\SS^1\bigr)$ be a
  neighborhood of a plastikstufe $\DOT\times \SS^1$ in
  $\bigl( \RR^3\times T^*\SS^1, \aOT + \lcan\bigr)$.

  For any arbitrarily large $C>0$ it is possible to embed the
  hypersurface
  \begin{equation*}
    S_C := \DD^2_{\le \pi} \times \bigl(- C,  C\bigr) \times
    \bigl(- \tfrac{\delta}{2}, \tfrac{\delta}{2}\bigr)
  \end{equation*}
  into $B(\epsilon) \times \DD_{<\delta}\bigl(T^*\SS^1\bigr)$ such
  that the contact structure induces the singular
  distribution~$\dD := \ker \bigl(r\sin r\, d\vartheta -t\,ds\bigr)$
  on $S_C$.  Here $(r,\vartheta)$ are polar coordinates on the disk,
  and $(s,t)$ are the natural coordinates on the rectangle
  $\bigl(- C, C\bigr) \times \bigl(- \tfrac{\delta}{2},
  \tfrac{\delta}{2}\bigr)$.
\end{lemma}
\begin{proof}
  Define for any $\hbar > 0$ an embedding
  \begin{equation*}
    \begin{split}
      \Psi_\hbar \colon&  \DD^2_{\le \pi} \times \RR^2 \to  \RR^3 \times T^*\SS^1 \\
      &\bigl(r,\vartheta;\; s, t\bigr) \mapsto \bigl(r,\vartheta,
      \hbar \, s; \; q = e^{is}, \; p=t + \hbar \, \cos r\bigr)
    \end{split}
  \end{equation*}
  ``reeling up'' the hypersurface along the $\SS^1$-core of the
  plastikstufe.  One easily verifies that
  \begin{equation*}
    \Psi_\hbar^* \bigl(\aOT + \lcan\bigr) = r\sin r\, d\vartheta
    -t\,ds
  \end{equation*}
  so that the singular distribution induced by
  $\ker\bigl(\aOT + \lcan\bigr)$ is indeed equal to $\dD$.

  Choose $\hbar = \epsilon/C$ and suppose that
  $\hbar < \frac{\delta}{2}$, we see that $\Psi_\hbar(S_C)$ lies in
  the given neighborhood
  $B(\epsilon) \times \DD_{<\delta}\bigl(T^*\SS^1\bigr)$.
\end{proof}

For general dimensions, the embedding is only slightly more
complicated.  Denote the standard coordinates on $\TT^n$ by
$\bfq = (q_1,\dotsc,q_n)$ and those on $T^*\TT^n$ by
$(\bfq,\bfp) = (q_1,\dotsc,q_n;p_1,\dotsc,p_n)$.

\begin{lemma}\label{lemma: unwrap torus}
  Any neighborhood
  $B(\epsilon) \times \DD_{<\delta}\bigl(T^*\TT^n\bigr)$ of a
  plastikstufe $\DOT\times \TT^n$ with $\epsilon > 0$ and $\delta > 0$
  arbitrarily small contains an embedded hypersurface of the form
  \begin{equation*}
    S_C := \DD^2_{\le \pi} \times
    \Bigl\{(s_1,\dotsc,s_n;t_1,\dotsc,t_n)\in \RR^{2n}  \Bigm| \;
    \abs{s_j} <  C, \; \abs{t_j} < \tfrac{\delta}{2\sqrt{n}}\Bigr\}
  \end{equation*}
  with $C>0$ arbitrarily large such that the contact structure induces
  the singular distribution
  $\dD = \ker \bigl(r\sin r\, d\vartheta - \sum_{j=1}^n
  t_j\,ds_j\bigr)$ on $S_C$.
\end{lemma}
\begin{proof}
  Choose constants $\hbar_1,\dotsc,\hbar_n > 0$ that are linearly
  independent over $\QQ$, and define a map
  $\Psi \colon \DD^2_{\le \pi} \times \RR^{2n} \to \RR^3 \times
  T^*\TT^n$ by 
  \begin{multline*}
    \bigl(r,\vartheta;\; s_1,\dotsc,s_n;t_1,\dotsc,t_n\bigr) \mapsto
    \Bigl(r,\vartheta, \; z = \sum_{j=1}^n \hbar_j \, s_j; \quad q_1 =
    e^{is_1}, \dotsc, q_n = e^{is_n}; \\ p_1 = t_1 + \hbar_1 \, \cos
    r, \dotsc , p_n = t_n + \hbar_n \, \cos r\Bigr)\;.
  \end{multline*}
  It is easy to verify that
  $\Psi^*\bigl(\aOT + \lcan\bigr) = r\sin r\, d\vartheta - \bigl(t_1
  \, ds_1 + \dotsm + t_n\, ds_n\bigr)$ induces the distribution~$\dD$
  on $S_C$.  It is also immediately clear that $\Psi$ is an immersion.

  \medskip

  To see that $\Psi$ is injective, use first that the images of two
  points $(r,\vartheta;\; s_1,\dotsc,s_n;t_1,\dotsc,t_n)$ and
  $(r',\vartheta';\; s_1',\dotsc,s_n';t_1',\dotsc,t_n')$ by $\Psi$ can
  only agree if $r = r'$, $\vartheta = \vartheta'$, and $t_j = t_j'$
  for all $j= 1,\dotsc, n$, and if $s_j - s_j'$ is for every
  $j=1,\dotsc,n$ an integer multiple of $2\pi$.  The equation
  $\hbar_1 \, s_1 + \dotsm + \hbar_n \, s_n = \hbar_1\, s_1' + \dotsm
  + \hbar_n \, s_n'$ implies that
  $\hbar_1\, (s_1 - s_1') + \dotsm + \hbar_n \, (s_n - s_n') = 0$, but
  by our assumption that the $\hbar_j$ are linearly independent over
  $\QQ$ it follows that all coefficients~$s_1 - s_1'$ need to vanish
  so that $\Psi$ is injective.

  \medskip
  
  We still need to verify that the image of $\Psi$ lies in the
  neighborhood $B(\epsilon) \times \DD_{<\delta}\bigl(T^*\TT^n\bigr)$.
  To respect the $z$-height, it suffices to choose
  $\hbar_1+\dotsm+\hbar_n < \epsilon / C$, so that the $z$-coordinate
  of $\Psi$ is bounded by $\epsilon$.  For the radius of the fibers in
  $\DD_{<\delta}\bigl(T^*\TT^n\bigr)$ choose
  $\hbar_j < \frac{\delta}{2\sqrt{n}}$ so that $\Psi$ also stays
  inside the $\delta$-disk bundle of $T^*\TT^n$.
\end{proof}

\section{The standard overtwisted contact structure on $\RR^3$}
\label{sec: overtwisted dim 3}

For a cylindrical box of height~$h$ around the standard overtwisted
disk $\DOT$ in $(\RR^3, \aOT)$ of the form
\begin{equation*}
  B(h) := \DD^2_{< \pi + \delta} \times (-h,h) \;,
\end{equation*}
it is well-known that the choice of $h$ is not relevant for the
contactomorphism type.  Below we will give a contact vector field that
can be used to prove this fact by hand, but instead one can also
easily convince oneself that all $B(h)$ are \emph{overtwisted at
  infinity} which uniquely characterizes by \cite{EliashbergContactR3}
a contact structure on $\RR^3$.

The main technical problem that we will deal with in this article is
to show that the choice of the $h$-parameter also remains largely
irrelevant for the contactomorphism type when we take the product with
a Liouville domain.

\medskip

We will now discuss a contact vector field~$X$ whose flow compresses
any large box~$B(h)$ into an arbitrarily small neighborhood of the
standard overtwisted disk.  Ideally we would like $X$ to be a strict
contact vector field or at least to have a constant scaling factor~$c$
such that $\lie{X}\aOT = c\cdot \aOT$.  Unfortunately such a vector
field cannot exist: firstly, $X$ should be contracting and thus it
needs to reduce the total volume.  This implies that $c$ would have to
be strictly negative on a predominant part of its domain.  On the
other hand, $c$ cannot be \emph{everywhere} strictly negative as this
would allow us to squeeze with the strategy of Section~\ref{sec:
  squeeze in high dimensions} a high-dimensional overtwisted chart
into an arbitrarily ``thin'' set, thus contradicting the existence of
tight contact manifolds.

The following vector field arose in discussions with Patrick Massot
around 2010,
\begin{equation}
  X := - z\,\partial_z  -
  \frac{r \cos(r) \sin(r)}{r + \cos(r) \sin(r)}\, \partial_r
  \label{eq: contact vf RR3}
\end{equation}
is well-defined and induces a contact flow on $(\RR^3, \aOT)$.

Even though this vector field might at first appear overly
complicated, note that all coordinates in $X$ are uncoupled.  This
allows us to see that its time~$T$ flow preserves the
$\vartheta$-coordinate, and it contracts the $z$-coordinate by the
factor~$e^{-T}$.  The radial coordinate is fixed on every cylinder of
radius $r\in \frac{\pi}{2} \, \NN$.  The cylinders of radius
$r\in \frac{\pi}{2} + \pi \, \NN$ are repelling; the cylinders of
radius $r\in \pi \, \NN$ are attracting, in the sense that all points
between these cylinders are pushed away from the repelling cylinder
towards the attracting one, see Fig~\ref{fig:contact vector field for
  aOT}.

\begin{figure}[htbp]
  \centering
  \includegraphics[height=6cm,keepaspectratio]{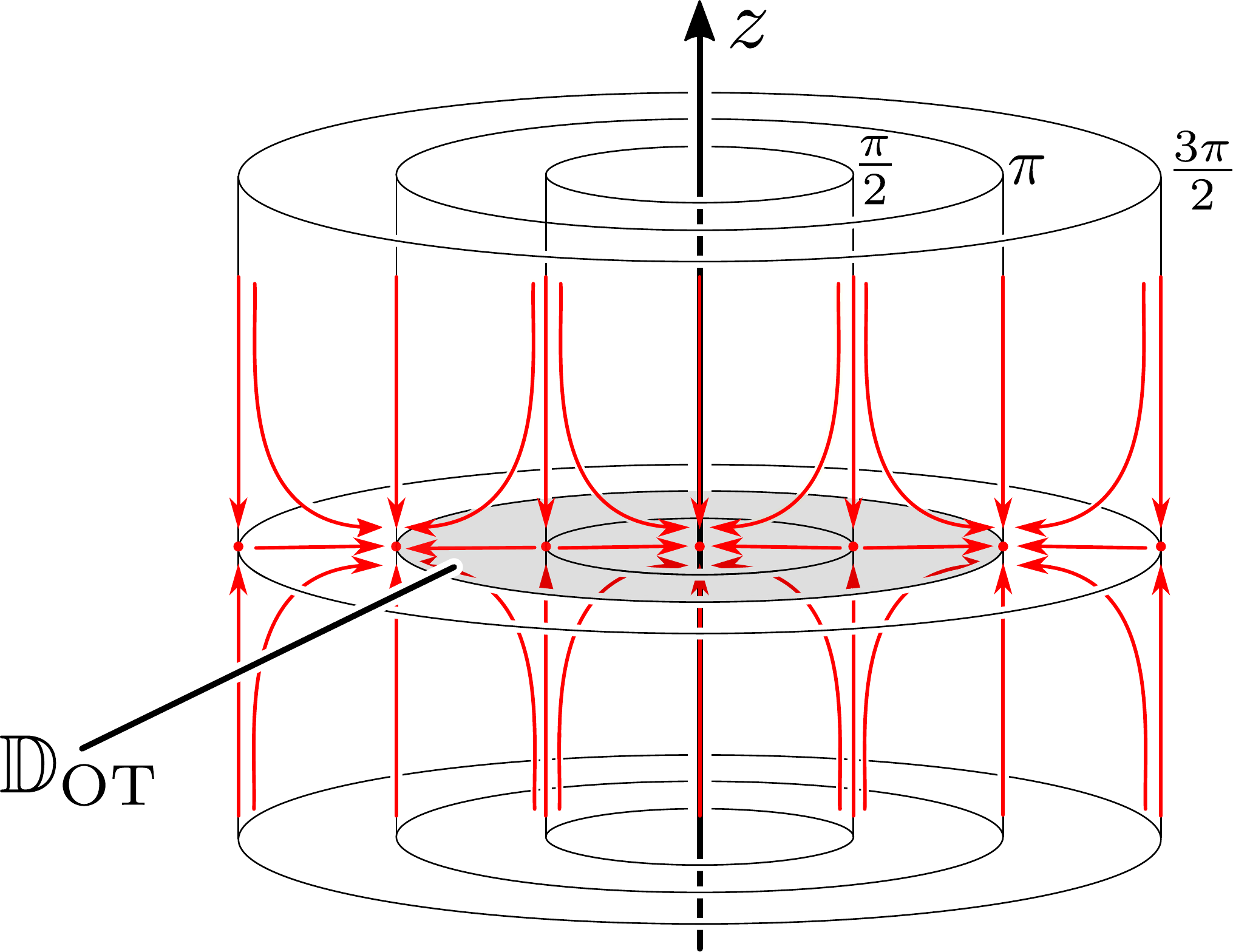}
  \caption{The flow of $X$ preserves the cylinders of radius
    $r\in \frac{\pi}{2} + \pi \, \NN$.  Note that the boundary of the
    standard overtwisted disk sits on an attracting
    cylinder. \label{fig:contact vector field for aOT}}
\end{figure}

\medskip

The height of the box~$B(h)$ is squeezed by the flow of $X$ by an
exponential factor, while the radial direction also shrinks, without
becoming ever smaller than $\pi$ though.  This way an arbitrarily tall
box~$B(h)$ can be squeezed by a contactomorphism into an arbitrarily
small neighborhood of the standard overtwisted disk.  In fact, one can
convince oneself that the image of a box~$B(h)$ for a certain choice
of $\delta > 0$ will be of the form
$\DD^2_{< \pi + \delta'} \times (-h',h')$ for some smaller $h'> 0$ and
$\delta' > 0$.

Finally note that $\lie{X}\aOT = g\cdot \aOT$ where
\begin{equation}
  g(r,\vartheta,z) =
  - \frac{\cos r\, (r\cos r + \sin r)}{r + \cos r \sin r} \;.
  \label{eq: scaling factor for vf on RR3}
\end{equation}
As we claimed above, $g$ takes both positive and negative values.
More precisely:

\begin{lemma}\label{lemma: behavior h}
  The function~$g\colon B(h) \to \RR$ is everywhere negative on $B(h)$
  except for the domain lying between $r_m = \pi/2$ and
  $r_M\approx 2.03$ such that $r_M =-\tan r_M$.  See also the graph in
  Fig.~\ref{fig:graph function g}.
\end{lemma}
\begin{proof}
  The function~$g$ only depends on the radial coordinate~$r$.  Its
  denominator is everywhere positive while the numerator changes once
  its sign at $r = \pi/2$, where $\cos r$ vanishes, and then again at
  $r\approx 2.03$, where $r\cos r + \sin r$ vanishes.
\end{proof}

We can also read off from the graph, Fig.~\ref{fig:graph function g},
that $g$ is everywhere smaller than $0.1$; i.e., even though $g$
becomes positive it only becomes very slightly so.

\begin{figure}[htbp]
  \centering
  \includegraphics[height=5cm,keepaspectratio]{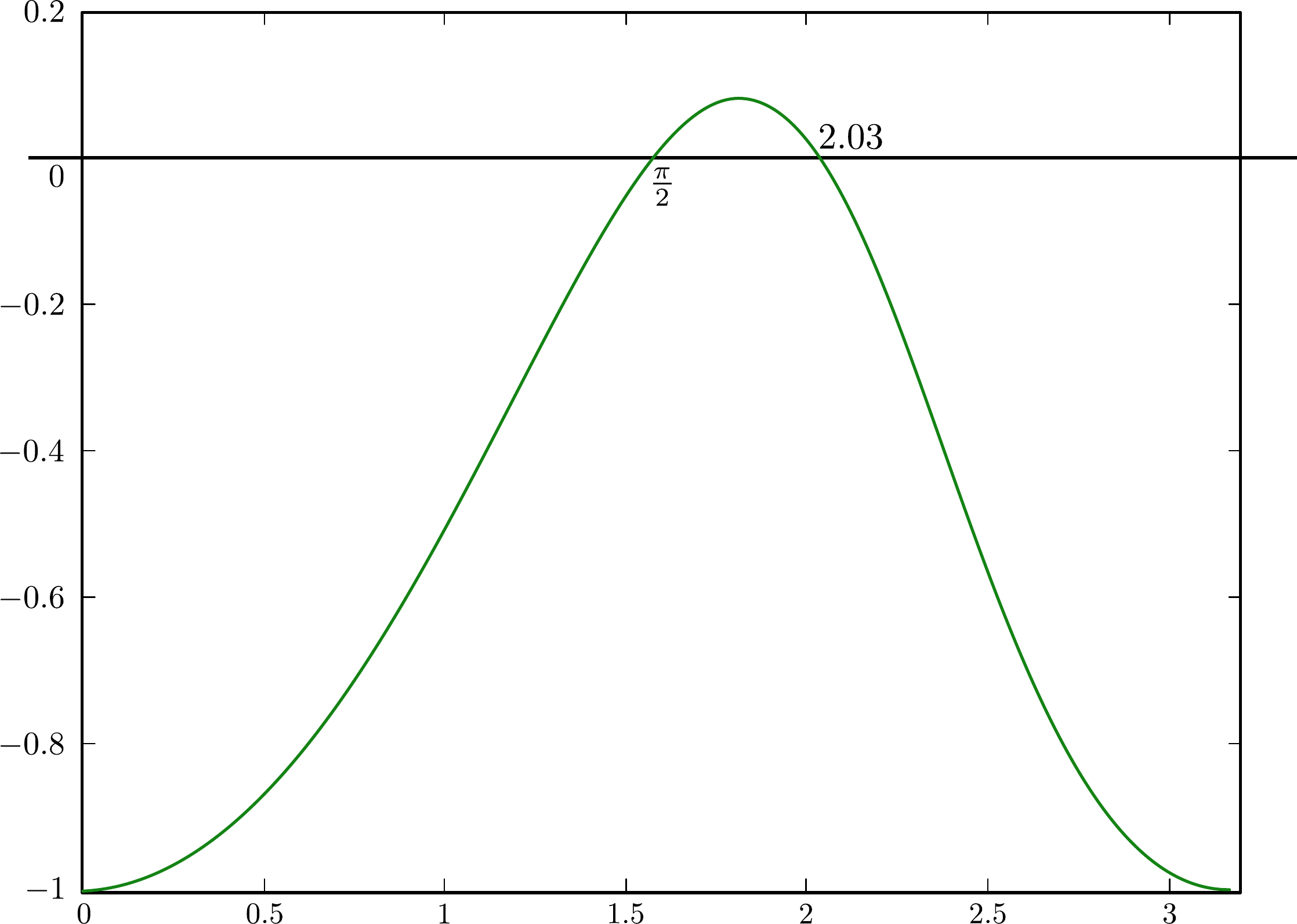}
  \caption{The graph of
    $g = - \frac{\cos r\, (r\cos r + \sin r)}{r + \cos r \sin
      r}$ \label{fig:graph function g}}
\end{figure}

\section{Contactomorphism on product structure}
\label{sec: squeeze in high dimensions}

Let $(M,\xi)$ be a contact manifold with contact form~$\alpha$.
Assume that $X$ is a contact vector field such that
$\lie{X}\alpha = g\cdot \alpha$ for some function $g\colon M \to \RR$.

Choose an exact symplectic manifold~$(W,d\lambda)$ that has a
Liouville vector field~$Y$, then we easily check that the contact form
of $\bigl(M\times W, \alpha + \lambda\bigr)$ is preserved by the
vector field
\begin{equation*}
  \hat X = X + g\cdot Y \;.
\end{equation*}

Note that even if $Y$ points outwards and is expanding on
$(W,d\lambda)$, the behavior of $\hat X$ on the product
manifold~$M\times W$ is controlled in $W$-direction by the sign of the
function~$g$ that might take positive or negative values.

\medskip

We consider now the main example we will be interested in: Let
$B(h) \subset \bigl(\RR^3,\aOT)$ be a box of height~$h$ as defined in
\eqref{eq: box}, and let
$\bigl(L, \;\langle \cdot,\cdot \rangle \bigr)$ be a Riemannian
manifold that does not need to be closed or geodesically complete.  We
denote the disk bundle of radius~$c$ in $\bigl(T^*L, d\lcan\bigr)$ by
\begin{equation*}
  \DD_{<c}\bigl(T^*L\bigr)  := \bigl\{ \nu\in T^*L\bigm|\;  \norm{\nu} < c\bigr\} \;.
\end{equation*}

\begin{proposition}\label{propo: compress height with little width}
  There exists a positive constant~$\mu_0 < 7/6$ such that every
  contact domain
  \begin{equation*}
    \Bigl(B(h) \,\times\, \DD_{<c}\bigl(T^*L\bigr), \;
    \ker\bigl(\aOT + \lcan\bigr) \Bigr)
  \end{equation*}
  with $c>0$ can be embedded by a contactomorphism into
  \begin{equation*}
    \Bigl(B(h') \,\times\, \DD_{<\mu_0\cdot c}\bigl(T^*L\bigr) \;,
	\ker\bigl(\aOT + \lcan\bigr) \Bigr)
  \end{equation*}
  independently of the choices of $h, h'>0$.
\end{proposition}
\begin{proof}
  For $h' \ge h$ the claim is obvious; for $h' < h$ the strategy is to
  use the flow of the vector field $\hat X := X + g\cdot Y$ with $X$
  and $g$ introduced in the previous section, and $Y$ the Liouville
  vector field on $T^*L$ that is defined by $\iota_Y d\lcan = \lcan$.

  By writing $Y$ in a coordinate chart, one easily convinces oneself
  that the time~$t$ flow of $Y$ simply consists of multiplying the
  fiber of $T^*L$ by $e^t$.  In particular, even if $L$ is open or has
  boundary there is no danger that the flow of $\hat X$ escapes
  transversely through the fibers of $T^*L$. Now let us study the
  behavior of the flow $\Phi^{\hat X}_T$ in more details.


\medskip

Recall that the coordinates are uncoupled by the flow of $X$ given in
\eqref{eq: contact vf RR3}.  We can thus write
  \begin{equation*}
    \Phi^X_T\bigl(r,\vartheta, z\bigr)
    = \bigl(F(r,T), \vartheta, e^{-T}\, z\bigr) \;,
  \end{equation*}
  where $F(r,T)$ is the solution of the ODE
  \begin{equation*}
    y'(t) = - \frac{y(t) \cos y(t) \sin y(t)}{y(t) + \cos y(t) \sin y(t)}\;,
    \quad \text{ and }\quad y(0) = r \;.
  \end{equation*}
  The flow~$\Phi^{\hat X}_T$ is therefore of the form
  \begin{equation*}
    \Phi^{\hat X}_T\bigl(r, \vartheta,z;\; \bfq,\bfp\bigr)
    = \Bigl(\Phi^X_T\bigl(r,\vartheta, z\bigr);\;
    \bfq, e^{G(r,T)}\cdot \bfp \Bigr) \;.
  \end{equation*}
  That is, the flow on the $\RR^3$-factor simply reduces to the
  corresponding flow of $X$ and can be evaluated independently of the
  $T^*L$-part; the flow on the cotangent bundle factor is obtained by
  multiplying the fiber direction by a positive function~$e^G$ that
  can be computed via
  \begin{equation}
    G(r,t) =  \int_0^t g\bigl(\Phi^X_s(r,\vartheta, z)\bigr)\,  ds
    = \int_0^t g\bigl(F(r,s)\bigr)\,  ds \;.  \label{eq: fnc G}
  \end{equation}

  \medskip

  If $T$ is chosen to be $T = \ln \frac{h}{h'}$, it follows that
  $\Phi^{\hat X}_T$ squeezes the first factor of
  $B(h) \times \DD_{<c}\bigl(T^*L\bigr)$ into $B(h')$.  By
  Lemma~\ref{lemma: bound on G} below, $G(r,t) < \ln (7/6)$ for any
  point in $B(h)$ and any $t\ge 0$.  This implies as desired that the
  initial domain is squeezed into
  $B(h') \times \DD_{<\frac{7}{6}\cdot c}\bigl(T^*L\bigr)$.
\end{proof}

\begin{lemma}\label{lemma: bound on G}
  The function~$G(r,t)$ given in \eqref{eq: fnc G} is bounded from
  above by $\ln(7/6)$ for all $r\in [0, \pi +\delta)$ and all
  $t\in [0,\infty)$.
\end{lemma}

The sharp upper bound in the lemma is
$\ln \bigl(\frac{2 r_M \sin r_M}{\pi} \bigr)$ with $r_M$ specified in
Lemma~\ref{lemma: behavior h}.

\begin{proof}
  Denote the $r$-coefficient of the vector field~$X$ by
  \begin{equation*}
    f(r) = - \frac{r \cos(r) \sin(r)}{r + \cos(r) \sin(r)} \;.
  \end{equation*}
  Then $F(r,t)$ is the flow of the field
  $X_r (r) := f(r)\, \partial_r$ on $[0, \pi + \delta)$, that is, $F$
  is the solution of the ordinary differential equation
  $\partial_t F(r,t) = f\bigl(F(r,t)\bigr)$ with initial condition
  $F(r,0) = r$.

  The only critical points of $X_r$ are the points
  $r\in \frac{\pi}{2} \, \NN$; see also Fig~\ref{fig:contact vector
    field for aOT}.  Furthermore, recall that $r = \frac{\pi}{2}$ and
  $r = \frac{3\pi}{2}$ are repelling, and that $r=\pi$ is an
  attracting critical point.
 
  According to Lemma~\ref{lemma: behavior h}, the function~$g$ is
  everywhere on $[0, \pi +\delta)$ negative except for the interval
  $[r_m,r_M]$ with $r_m = \frac{\pi}{2}$, and $r_M \approx 2.03$ given
  by $r_M = -\tan r_M$.

  Since all trajectories of $X_r$ starting in
  $\bigl[0, \frac{\pi}{2}\bigr]$ are trapped inside this interval, the
  function~$G(r,t)$ will be negative for all
  $r\in \bigl[0, \frac{\pi}{2}\bigr]$ and all $t\ge 0$.  Similarly,
  the points in $[\pi,\pi+\delta)$ are pulled by the flow towards
  $r = \pi$ without ever crossing this point.  Thus $G(r,t)$ will also
  be negative for all $r\in [\pi,\pi+\delta)$ and all $t \ge 0$.

  \smallskip
  
  It only remains to understand the behavior of $G(r,t)$ for
  $r\in \bigl(\frac{\pi}{2}, \pi\bigr)$.  Since $f$ is strictly
  positive on this interval, it follows that, for every initial value
  $r\in \bigl(\tfrac{\pi}{2}, \pi\bigr)$,
  \begin{equation*}
    F(r,\cdot)\colon \RR \to \bigl(\tfrac{\pi}{2}, \pi\bigr)
  \end{equation*}
  is an orientation preserving diffeomorphism, and in particular there
  is a unique time $T_r\in \RR$ such that $F(r,T_r) = r_M$.

  For every fixed $r\in \bigl(\tfrac{\pi}{2}, r_M\bigr]$ and all
  positive $t\ge 0$, the upper bound of $G(r,t)$ in \eqref{eq: fnc G}
  is then given by
  \begin{equation*}
    G(r,T_r) =  \int_0^{T_r} g\bigl(F(r,s)\bigr)\,  ds\;,
  \end{equation*}
  because $g$ is strictly positive up to $t = T_r$ so that
  $G(r,\cdot)$ increases up to that moment; for all later times
  $t > T_r$, the trajectory $F(r,t)$ lies in the zone $[r_M,\pi)$
  where $g$ is negative so that $G(r,t) \le G(r,T_r)$ for all
  $t \ge T_r$.

  \medskip
  
  To compute $G(r,T_r)$ use that
  $F(r,\cdot)\colon \RR \to (\pi/2, \pi)$ is for every choice of
  $r\in \bigl(\tfrac{\pi}{2}, \pi\bigr)$ a diffeomorphism, so that we
  can substitute $u = F(r,s)$ in the integral using that
  $\frac{du}{ds} = f\bigl(F(r,s)\bigr) = f(u)$, and obtain
  \begin{equation*}
    \begin{split}
      G(r,T_r) &= \int_0^{T_r} g\bigl(F(r,s)\bigr)\, ds =
      \int_{F(r,0)}^{F(r,T_r)} \frac{g(u)}{f(u)}\, du =
      \int_r^{r_M} \frac{u\cos u + \sin u}{u \sin u}\, du  \\
      &= \Bigl.\ln\bigl(u \sin u\bigr)\Bigr|_{r}^{r_M} = \ln \bigl(
      \frac{r_M \sin r_M}{r \sin r} \bigr) \;.
    \end{split}
  \end{equation*}
  The denominator $r \sin r$ is increasing on
  $\bigl[\tfrac{\pi}{2}, r_M]$ so that its smallest value on this
  interval is attained at $r = \tfrac{\pi}{2}$.  We obtain the
  estimate
  \begin{equation*}
    G(r,t)  \le \ln \bigl( \frac{r_M \sin r_M}{r \sin r} \bigr)
    < \ln \bigl(\frac{r_M \sin r_M}{\tfrac{\pi}{2} \, \sin \tfrac{\pi}{2}} \bigr)
    =  \ln \bigl(\frac{2 r_M \sin r_M}{\pi} \bigr) < \ln (7/6)
  \end{equation*}
  finishing the proof.
\end{proof}

In particular, we obtain the following result.

\begin{corollary}\label{cor: can embedded high domain close to
    hypersurface}
  Let $L$ be a manifold that does not need to be closed, and let
  \begin{equation*}
   \Sigma:= \DD^2_{\le \pi} \,
    \times\, \DD_{<c_0}\bigl(T^*L\bigr)
  \end{equation*}
  be a hypersurface in a contact manifold~$(M,\xi)$ such that the
  singular distribution induced by $\xi$ on the hypersurface agrees
  with the kernel of the $1$-form
  $\beta = r\sin r\, d\vartheta + \lcan$, where $(r,\vartheta)$
  denotes the polar coordinates on $\DD^2_{\le \pi}$.

  \smallskip
  
  Let $c > 0$ be such that $\mu_0 c < c_0$ with the constant
  $\mu_0 < 7/6$ in Proposition~\ref{propo: compress height with little
    width}.  Then we can embed the contact domain
  \begin{equation*}
    \Bigl(B(h) \,\times\, \DD_{<c}\bigl(T^*L\bigr), \;
    \ker\bigl(\aOT + \lcan\bigr) \Bigr)
  \end{equation*}
  of ``width''~$c>0$ and of any chosen ``height''~$h>0$ into an
  arbitrarily small neighborhood of the hypersurface
  $\DD^2_{\le \pi} \, \times\, \DD_{<c_0}\bigl(T^*L\bigr)$.
\end{corollary}
\begin{proof}
  The induced singular distribution of a hypersurface determines by
  Proposition~\ref{contact germ determined by hypersurface} the germ
  of the contact structure on a neighborhood of the hypersurface.  We
  can thus assume that $\Sigma$ has a neighborhood~$U$ that is
  contactomorphic to
  $\bigl(B(\epsilon) \,\times\, \DD_{<c_0}\bigl(T^*L\bigr), \;
  \ker\bigl(\aOT + \lcan\bigr)\bigr)$ for small $\epsilon > 0$.  By
  Proposition~\ref{propo: compress height with little width} we can
  thus embed $B(h) \,\times\, \DD_{<c}\bigl(T^*L\bigr)$ into $U$.
\end{proof}

\appendix

\section{The Legendrian unknot is loose in a sufficiently large
  overtwisted chart}\label{sec: overtwisted then loose}

By now, it is well-known that Legendrian unknots are loose in ambient
manifolds containing a large neighborhood of an overtwisted chart.
The argument was first sketched in \cite{MurphyEtAlLoosePlastikstufe},
and Huang gave later a more detailed proof in \cite{HuangPlastistufe}.
Unfortunately though he uses piecewise smooth Legendrians so that a
lot of the potential clarity was lost.  In this appendix, we take a
new attempt to write down the proof.  Once a certain $3$-dimensional
result is accepted as a black-box, we show that the main step to pass
from dimension~$3$ to higher dimensions essentially boils down to a
careful inspection of the original definition of looseness given by
Murphy \cite{MurphyLoose}.

\bigskip

Let
$\bigl(\RR^{2n+1}, \xi_0 = \ker (dz - \sum_{j=1}^n y_j\,dx_j) \bigr)$
with coordinates $(\x,\y,z) = (x_1,\dotsc,x_n, y_1,\dotsc, y_n, z)$ be
the standard contact space.  The \defin{Legendrian unknot~$\Lambda_0$}
in $\RR^{2n+1}$ can be given by the embedding
\begin{equation}
  \SS^n \hookrightarrow (\RR^{2n+1},\xi_0), \;
  (\x,s) \mapsto  \bigl(\x,  - s\x, \tfrac{s^3}{3}\bigr) \;.
  \label{eq: standard Legendrian unknot}
\end{equation}
By extension, any Legendrian~$\Lambda$ in a contact manifold is called
a \defin{Legendrian unknot} if there exists a Darboux chart containing
$\Lambda$ in its interior such that $\Lambda$ agrees in the chart with
$\Lambda_0$.

\medskip

Let $(M,\xi)$ be a contact manifold.  We want to study Legendrians in
$M$ that look locally like product submanifolds in the following
sense: Suppose that there is an open subset~$U \subset M$ that is
diffeomorphic to $U_M \times U_W$ where $U_M$ is a manifold that
carries a contact form~$\alpha$, and $U_W$ is an open Liouville domain
with Liouville form~$\lambda$ such that
$\restricted{\xi}{U} = \ker (\alpha + \lambda)$.  Then assume that a
Legendrian~$\Lambda$ satisfies
\begin{equation*}
  \Lambda \cap \bigl(U_M\times U_W\bigr) = L \times N \;,
\end{equation*}
where $L$ is a Legendrian in $\bigl(U_M, \alpha\bigr)$ and $N$ is an
exact Lagrangian in $\bigl(U_W, \lambda\bigr)$ with
$\restricted{\lambda}{TN} = 0$.  Note that we do not assume in general
that $L$ or $N$ are closed.

The key notion we want to study in this appendix is due to Murphy
\cite{MurphyLoose}.

\begin{definition}
  Let $\Lambda$ be a Legendrian in $(M,\xi)$ that is locally in
  product form $L \times Z_\rho$ in a chart
  $\bigl(C \times V_\rho, \ker(\alpha_0 + \lcan)\bigr)$, where
  \begin{itemize}
  \item $C = \{(x,y,z)\in \RR^3|\; x, y, z \in (-1,1) \}$ is a cube
    with side lengths~$2$ and standard contact form
    $\alpha_0 = dz - y\, dx$,
  \item
    $V_\rho = \{(\bfq,\bfp)\in \RR^{2n-2}|\; \|\bfq\| < \rho, \|\bfp\|
    < \rho \}$ with Liouville form $\lcan = -\sum_j p_j\, dq_j$.
  \item $L$ is a properly embedded Legendrian arc whose front is a
    zig-zag and which is equal to the set $\{y=z=0\}$ near the
    boundary.
  \item $Z_\rho= \bigl\{ (\bfq,\bfp)\in V_\rho|\; \bfp= \0 \bigr\}$.
  \end{itemize}
  We say that $\Lambda$ is \defin{loose} if $\rho > 1$.
\end{definition}

\begin{remark}
  Note that if we replace the cube~$C$ in the definition above with
  any cube of side lengths smaller than $2$, then it can be seen with
  the argument in \cite[Proposition~4.4]{MurphyLoose} that the
  corresponding Legendrian is still loose.
\end{remark}

\medskip

The result we want to show in this appendix is the following
proposition.

\begin{proposition}
  There exists a $\rho_0 >0$ (that is independent of the dimension of
  $V_\rho$) such that the Legendrian unknot~$\Lambda_0$ is loose in
  every contact manifold
  \begin{equation*}
    \bigl(B(1)\times V_\rho\,,\; \ker(\aOT + \lcan)\bigr)
  \end{equation*}
  for which $\rho>\rho_0$.  Here,
  $\aOT = \cos(r)\,dz + r\sin(r)\, d\vartheta$ denotes the standard
  overtwisted contact form on $B(1)$, see also Section~\ref{sec:
    overtwisted dim 3}.
\end{proposition}

\begin{proof}
  By \cite{Dymara} or \cite{EtnyreOvertwisteKnots}, see also the
  details in \cite{MurphyEtAlLoosePlastikstufe}, the Legendrian unknot
  is in any overtwisted $3$-manifold the stabilization of another
  Legendrian knot~$L_1$.  More precisely, let $\bigl(B(1), \aOT\bigr)$
  be the cylindrical box surrounding an overtwisted disk described in
  Section~\ref{sec: overtwisted dim 3}.  We can assume that there is a
  Darboux chart~$U_1$ centered around a point of $L_1$ such that
  \begin{itemize}
  \item the restriction of $\aOT$ agrees in the coordinates of the
    Darboux chart with the standard form $dz - y\,dx$,
  \item $U_1$ is a cube of size~$\epsilon_1 < 1$,
  \item $L_1\cap U_1$ is the Legendrian arc $\{y=z=0\}$.
  \end{itemize}
  We stabilize $L_1$ inside $U_1$ by adding a zig-zag.  The resulting
  knot is then a Legendrian unknot~$L_0$.

  In particular, there exists a Darboux chart~$U_0$ in $B(1)$ such
  that $L_0$ lies in standard position~\eqref{eq: standard Legendrian
    unknot} inside $U_0$.  The restriction $\restricted{\aOT}{U_0}$
  with respect to the coordinates of $U_0$ will be of the form
  $e^{f(x,y,z)} \, \bigl(dz - y\,dx\bigr)$ for some smooth function
  $f$ that probably cannot be chosen to be equal to $0$.  Nonetheless,
  we can assume that there are constants~$c_0 >0$, and $C_0 > 1$ such
  that $c_0\le e^f\le C_0$.

  \medskip

  Using the Darboux chart~$U_0$ in $B(1)$, we can find in the product
  contact manifold
  \begin{equation*}
    \bigl(B(1)\times V_\rho\,,\; \ker(\aOT + \lcan)\bigr)
  \end{equation*}
  a higher dimensional Darboux chart $U_0\times V_{\rho'}$ for
  $\rho' = \frac{\rho}{C_0}$ embedded by
  \begin{equation*}
    U_0\times V_{\rho'}\to U_0\times V_\rho, \;
    (x,y,z; \bfq,\bfp) \mapsto \bigl(x,y,z; \bfq, e^{f(x,y,z)} \,\bfp\bigr)\;.
  \end{equation*}

  If $\rho' > 1$, we can embed the Legendrian unknot~$\Lambda_0$ into
  this Darboux chart using the standard map~\eqref{eq: standard
    Legendrian unknot}
  \begin{equation*}
    \SS^{n+1} \hookrightarrow \bigl(U_0\times V_{\rho'},\;\xi_0\bigr), \;
    (x_0, \x, s) \mapsto
    \bigl(x_0,  - s x_0, \tfrac{s^3}{3}; \x, - s \x \bigr) \;.
  \end{equation*}
  Note that the intersection of $\Lambda_0$ with the slice
  $B(1)\times\{\0\}$ is precisely the unknot~$L_0$ in $U_0$ that we
  had singled out in $B(1)$.

  By slightly deforming $\Lambda_0$ close to $L_0\times \{\0\}$ we can
  assume that $\Lambda_0$ is locally in product form
  $L_0\times Z_\epsilon$ for some small constant $\epsilon >0$.  This
  can be easily seen using the front projection (even tough ``seeing''
  the front projection requires from dimension~$7$ on some
  experience).  More explicitly, let $g\colon [0,1]\to [0,1]$ be a
  monotonous smooth function such that $g$ is equal to $0$ in a
  neighborhood of $0$ and $1$ in a neighborhood of $1$.  We can then
  consider the deformed sphere $\SS_g^{n+1} \subset \RR^{n+2}$ given
  by the equation
  \begin{equation*}
    x_0^2 + s^2 + g\bigl(\x^2\bigr) \cdot \x^2  = 1 \;.
  \end{equation*}
  Interpolating linearly between the equation of the round sphere and
  the deformed one, one sees that both are isotopic to each other.

  We define a Legendrian embedding of $\SS_g^{n+1}$ by
  \begin{equation*}
    \SS_g^{n+1} \hookrightarrow \bigl(U_0\times V_{\rho'},\;\xi_0\bigr), \;
    (x_0, \x, s) \mapsto
    \Bigl(x_0,  - s x_0, \tfrac{s^3}{3}; _; \x, - s \x\, \bigl(g(\x^2) + \x^2\, g'(\x^2) \bigr) \Bigr) \;.
  \end{equation*}
  We denote this deformed sphere by $\Lambda_0'$.  It is Legendrian
  isotopic to the initial Legendrian unknot, and it is composed of a
  cylindrical part $L_0\times Z_\delta$ for small values of $\x$ in
  $U_0\times V_\delta$.

  \smallskip
  
  Recall that the Darboux chart $U_1\times V_{\rho'}$ had been
  embedded into $B(1)\times V_\rho$ via the map
  $(x,y,z; \bfq,\bfp) \mapsto \bigl(x,y,z; \bfq, e^{f(x,y,z)}
  \,\bfp\bigr)$.  By looking at the preimage of this embedding, it
  follows that $\Lambda_0'$ also has a cylindrical segment in
  $B(1)\times V_\rho$.  More explicitly, we have shown that after an
  isotopy the Legendrian unknot~$\Lambda_0$ in $B(1)\times V_\rho$ has
  a cylindrical part of the form $L_0\times Z_{\delta'}$ in the open
  ball $B(1)\times V_{\delta'}$ with $\delta'= \delta/c_0$.

  We stretch out $\Lambda_0'$ in the $V_\rho$-direction of
  $B(1)\times V_\rho$ using an isotopy of the form
  $(x,y,z; \bfq, \bfp) \mapsto \bigl(x,y,z; \; e^t\,\bfq,
  e^{-t}\,\bfp\bigr)$ where the maximal size of $t \ge 0$ depends on
  the width of $V_\rho$.  If there is enough space, then we can
  suppose that the cylindrical part found above expands to be of the
  form $L_0\times Z_\eta$ in the open ball $B(1)\times V_\eta$ with
  $\eta > 1$.

  If we now consider the Darboux chart~$U_1 \subset B(1)$, we see that
  the intersection of the deformed Legendrian sphere with
  $U_1\times V_\eta$ is $L\times Z_\eta$, where $L$ is a Legendrian
  arc in $U_1$ whose front is a zig-zag, just as in the definition of
  looseness.  If $\eta > 1$ and if $U_1$ is a cube of size smaller
  than $1$ (see remark above following the definition of looseness),
  then the deformed unknot and thus also $\Lambda_0$ are loose.
\end{proof}

\section{Contact germ along a hypersurface}\label{sec: contact germ by
  hypersurface}

A folklore result states that a hypersurface in a contact manifold
determines the germ of the contact structure surrounding it.  Not
having found a proof for dimension $>3$ in the literature we have
decided to add it here.

\begin{proposition}\label{contact germ determined by hypersurface}
  Let $(M_0,\xi_0)$ and $(M_1,\xi_1)$ be two $(2n+1)$-dimensional
  contact manifolds, and let $\Sigma$ be a (not necessarily closed)
  $2n$-dimensional manifold.  Assume that there are two embeddings
  \begin{equation*}
    \iota_0\colon \Sigma \hookrightarrow M_0 \quad\text{ and }\quad
    \iota_1\colon \Sigma \hookrightarrow M_1 
  \end{equation*}
  such that the singular distributions $\dD_0 = (D\iota_0)^{-1}(\xi_0)$
  and $\dD_1 = (D\iota_1)^{-1}(\xi_1)$ agree.
  
  \medskip
  
  Then there exist a neighborhood $U_0\subset M_0$ of
  $\iota_0(\Sigma)$, a neighborhood $U_1\subset M_1$ of
  $\iota_1(\Sigma)$, and a contactomorphism
  \begin{equation*}
    \Phi\colon (U_0,\xi_0) \to (U_1,\xi_1)
  \end{equation*}
  such that $\Phi\circ \iota_0 = \iota_1$.
\end{proposition}

\begin{remark}
  To be able to apply Proposition~\ref{contact germ determined by
    hypersurface} to a hypersurface $\Sigma$ with non-empty boundary,
  one needs to attach a small collar along $\p \Sigma$, and extend the
  embeddings $\iota_0$ and $\iota_1$ in such a way that the singular
  distributions $\dD_0 = (D\iota_0)^{-1}(\xi_0)$ and
  $\dD_1 = (D\iota_1)^{-1}(\xi_1)$ agree.
\end{remark}

We split the proof of Proposition~\ref{contact germ determined by
  hypersurface} into several lemmas.  The first one is due to Giroux,
but we learned about it from \cite{MassotThesis}.

\begin{lemma}\label{scaling of defining 1form of sing distribution}
  Let $\Sigma$ be a (not necessarily closed) manifold carrying a
  (cooriented) singular distributions~$\dD$ that is given as the
  kernel of a $1$-form~$\beta$ such that $d\beta$ does not vanish at
  the singular points of $\dD$; i.e., at the points where $\beta = 0$.

  If $\beta'$ is any other $1$-form such that $\dD = \ker \beta'$
  inducing the same coorientation, and such that $d\beta'$ does not
  vanish either at the singular points of $\dD$, then there exists a
  smooth positive function $f\colon \Sigma\to ]0,\infty[$ such that
  \begin{equation*}
    \beta = f\cdot \beta' \;.
  \end{equation*}
\end{lemma}
\begin{proof} 
  Denote the set of all regular points of the distributions~$\dD$ by
  $\Ureg = \{p \in \Sigma|\; \dD_p \ne T_p\Sigma\}$.  On $\Ureg$, we
  can simply define $f$ to be the quotient $\beta(X)/\beta'(X)$, where
  $X$ is any vector field on $\Ureg$ that is transverse to $\dD$.  We
  are thus left with studying the singular points $p \notin \Ureg$ of
  $\dD$, where $\beta$ and $\beta'$ both vanish, and proving that $f$
  extends to a non-vanishing smooth function.

  Use a coordinate chart for $\Sigma$ centered at
  $p\in \Sigma \setminus \Ureg$ with coordinates
  $\x = (x_1, \dotsc,x_n)$.  We can then write
  \begin{equation*}
    \beta = g_1\, dx_1 + \dotsm + g_n\, dx_n
    \quad\text{ and }\quad
    \beta' = g_1'\, dx_1 + \dotsm + g_n'\, dx_n
  \end{equation*}
  with functions $g_1,\dotsc,g_n$ and $g_1',\dotsc,g_n'$ such that all
  $g_j$ and all $g_j'$ vanish at the origin.  In fact for each $j$,
  the two functions~$g_j$ and $g_j'$ vanish precisely on the same
  subset.  By our assumption $d\beta' \ne 0$ at $p$, so that we can
  assume after possibly permuting the coordinates that
  $\frac{\partial g_1'}{\partial x_2} (\0) \ne 0$.

  We will now show that $f$ extends in the chart smoothly to a
  neighborhood of the origin such that
  $g_1(\x) = f(\x) \cdot g_1'(\x)$.  Note that
  $\{\x|\; g_1'(\x) \ne 0\}$ is a subset of $\Ureg$ so that $f$ is a
  well-defined function on this subset.  The condition
  $\frac{\partial g_1'}{\partial x_2} (\0) \ne 0$ allows us to apply
  the implicit function theorem to find a new set of coordinates
  $\y = (y_1,\dotsc,y_n)$ for which $g_1'$ simplifies to
  $g_1'(\y) = y_2$.  In this new chart, we obtain that $f$ is defined
  in particular for all points $\{y_2 \ne 0\}\subset \Ureg$.

  Consider now the function~$g_1$ represented with respect to the
  $\y$-coordinates.  It also vanishes precisely along the hyperplane
  $\{y_2=0\}$ so that there exists a smooth functions~$\tilde g_1$
  allowing us to write $g_1$ as
  \begin{equation*}
    g_1(\y) = y_2\, \tilde g_1(\y) \;,
  \end{equation*}
  see for example \cite[Lemma~2.1]{Milnor_MorseTheory}.  Using this
  representation we see that $f(\y) = \tilde g_1(\y)$ extends to a
  smooth function on the whole chart so that it obviously satisfies
  the equation $g_1 = f \cdot g_1'$.
  
  In particular, since $\Ureg$ is dense in $\Sigma$ the continuous
  extension of $f$ is unique and does not depend on our choice of
  charts.  This way, $f$ can be defined smoothly on all of $\Sigma$,
  and it satisfies $\beta = f\cdot \beta'$ on $\Sigma$.

  It remains to prove that $f$ does not vanish anywhere, but this is
  clear because if $f$ ever vanished at a singular point~$p$ of $\dD$,
  we would find from $\beta = f\cdot \beta'$ that $d\beta = 0$ at $p$
  --- contrary to our assumption that $d\beta \ne 0$ along the
  singular set of $\dD$.
\end{proof}

\begin{lemma}\label{linear isomorphism around hypersurface}
  Let $(M_0,\xi_0)$ and $(M_1,\xi_1)$ be two $(2n+1)$-dimensional
  contact manifolds with contact forms~$\alpha_0$ and $\alpha_1$
  respectively, and let $\Sigma$ be  a (not necessarily closed)
  manifold of dimension~$2n$.

  Suppose that there are two embeddings
  \begin{equation*}
    \iota_0\colon \Sigma \hookrightarrow (M_0,\xi_0) \text{ and }
    \iota_1\colon \Sigma \hookrightarrow (M_1,\xi_1)
  \end{equation*}
  of $\Sigma$ into $M_0$ and $M_1$ such that
  $\iota_0^* \alpha_0 = \iota_1^* \alpha_1$.

  \smallskip

  Then, there is a bundle isomorphism
  \begin{equation*}
    \Phi\colon \restricted{TM_0}{\Sigma} \to \restricted{TM_1}{\Sigma}
  \end{equation*}
  such that:
  \begin{itemize}
  \item [(i)] $\restricted{\Phi}{T\Sigma} = \id_{T\Sigma}$ (We
    identify here, and in the rest of the proof, $T\Sigma$ with the
    tangent spaces of $\iota_0(\Sigma)$ and $\iota_1(\Sigma)$.)
  \item [(ii)] $\alpha_1\circ \Phi = \alpha_0$ on
    $\restricted{TM_0}{\Sigma}$.
  \item [(iii)] The linear interpolation
    $(1-\tau)\, d\alpha_0 + \tau\, (d\alpha_1\circ \Phi)$ is for every
    $\tau \in [0,1]$ a symplectic form on
    $\restricted{\xi_0}{\Sigma}$.
  \end{itemize}
\end{lemma}
\begin{proof}
  Denote the $1$-form $\iota_0^* \alpha_0 = \iota_1^* \alpha_1$ by
  $\beta$.  To construct the desired bundle isomorphism~$\Phi$, we
  distinguish two types of subsets of $\Sigma$: Define
  \begin{equation*}
    \Ureg = \bigl\{p\in \Sigma\bigm|\; \beta_p \ne 0 \bigr\}
    \quad \text{ and } \quad
    \Usymp = \bigl\{p\in \Sigma\bigm|\; d\beta_p^n \ne 0 \bigr\} \;.
  \end{equation*}
  Both sets are open and their union covers all of $\Sigma$,
  because $d\beta = \iota_j^* d\alpha_j$ is at every point
  $p\in \Sigma$ where $T_p\Sigma = \xi_j(p)$ a maximally
  non-degenerate form on $T_p\Sigma$; that is, $d\beta$ is a
  symplectic form on $T_p\Sigma$ at every point~$p$ where $\beta$
  vanishes.

  We construct now separate bundle isomorphisms over
  $\Ureg$ and over $\Usymp$ that we then glue
  together using a partition of unity.

  \medskip
  
  Over $\Usymp$, we can decompose $TM_j$ as
  $T\Sigma \oplus \Span(R_j)$, where $R_j$ is the Reeb vector field of
  $\alpha_j$.  This allows us to define a first bundle isomorphism
  \begin{equation*}
    \Phi_{\mathrm{symp}} \colon \restricted{TM_0}{\Usymp} \to
    \restricted{TM_1}{\Usymp}
  \end{equation*}
  by $\Phi_{\mathrm{symp}}(v + c R_0) = v + c R_1$ for every
  $v\in \restricted{T\Sigma}{\Usymp}$ and every $c\in \RR$.  It is
  easy to verify that $\alpha_0$ and
  $\alpha_1\circ \Phi_{\mathrm{symp}}$ agree on
  $\restricted{TM_0}{\Usymp}$.  Furthermore, $d\alpha_0$ and
  $d\alpha_1\circ \Phi_{\mathrm{symp}}$ also agree, because for any
  pair of vectors $v,v' \in T\Sigma$, we have
  $d\alpha_j(v,v') = d\beta(v,v')$ on one hand, and
  $d\alpha_j(R_j,v) = 0$ on the other, for both $j=0, 1$.

  \medskip

  To construct a bundle isomorphism over $\Ureg$, we define the
  \emph{characteristic foliation~$\fF$} on $\Sigma$.  It is
  characterized over $\Ureg$ as the (singular) subdistribution of
  $\ker \beta = T\Sigma \cap \xi$ on which
  $\restricted{d\beta}{\ker \beta}$ vanishes.  A dimension count shows
  that $\fF$ is of dimension~$1$.

  Choose a volume form~$d\vol_\Sigma$ on $\Sigma$, and let $X$ be the
  vector field determined by the equation
  \begin{equation*}
    \iota_X d\vol_\Sigma = \beta\wedge (d\beta)^{n-1} \;.
  \end{equation*}
  Since $X$ only vanishes at points where $\beta$ vanishes, it follows
  that $X$ is everywhere on $\Ureg$ non-singular, and it is easy to
  convince oneself that the characteristic foliation is generated by
  $X$.
  
  \smallskip
  
  Choosing compatible complex structures~$J_0$ on $(\xi_0,d\alpha_0)$
  and $J_1$ on $(\xi_1,d\alpha_1)$, we define two vector
  fields~$Y_0 = J_0 \cdot X$ and $Y_1 = J_1 \cdot X$ along
  $\iota_0(\Sigma)$ and $\iota_1(\Sigma)$ respectively.  These vector
  fields are everywhere over $\Ureg$ transverse to $\Sigma$ and they
  lie in the kernel of $\alpha_j$.  This way, we can split the tangent
  bundles as
  \begin{equation*}
    \restricted{TM_j}{\Ureg} =
    \restricted{T\Sigma}{\Ureg} \oplus
    \restricted{\Span(Y_j)}{\Ureg}  \;,
  \end{equation*}
  and use these decompositions to define the bundle isomorphism
  \begin{equation*}
    \Phi_{\mathrm{reg}} \colon \restricted{TM_0}{\Ureg} \to
    \restricted{TM_1}{\Ureg}
  \end{equation*}
  by $\Phi_{\mathrm{reg}}\bigl(v + c Y_0\bigr) = v + c Y_1$ for every
  $v\in \restricted{T \Sigma}{\Ureg}$ and every $c\in \RR$.  Again, we
  easily check that $\alpha_1\circ \Phi_{\mathrm{reg}}$ agrees on
  $\restricted{TM_0}{\Ureg}$ with $\alpha_0$ so that
  $\Phi_{\mathrm{reg}}\bigl(\restricted{\xi_0}{\Ureg}\bigr) =
  \restricted{\xi_1}{\Ureg}$.

  To understand the interpolation between $d\alpha_0$ and
  $d\alpha_1\circ \Phi_{\mathrm{reg}}$, choose at a point $p\in \Ureg$
  a basis of $\xi_0(p)$ of the form
  $v_1,\dotsc,v_{2n-2}, X(p), Y_0(p)$, where the $v_j$ all lie in
  $\ker \beta$ and are complementary to $X(p)$.  Assume they are
  ordered in such a way that
  $d\alpha_0^{n-1} (v_1,\dotsc,v_{2n-2}) =
  d\beta^{n-1}(v_1,\dotsc,v_{2n-2}) =
  d\alpha_1^{n-1}(v_1,\dotsc,v_{2n-2}) > 0$.  Note that
  $d\alpha_0(X,\cdot)$ and $d\alpha_1(X,\cdot)$ vanish on all the
  vectors $v_1,\dotsc,v_{2n-2}, X$.
  
  Define
  $d\alpha_\tau = (1-\tau)\, d\alpha_0 + \tau \, (d\alpha_1\circ
  \Phi_{\mathrm{reg}})$ for any $\tau\in[0,1]$.  Then we
  compute for all $\tau \in [0,1]$ that
  \begin{equation*}
    d\alpha_\tau^n (v_1,\dotsc,v_{2n-2}, X, Y_0) = n\, d\alpha_\tau
    (X, Y_0) \cdot d\alpha_\tau^{n-1} (v_1,\dotsc,v_{2n-2})
    > 0 \;,
  \end{equation*}
  because
  $d\alpha_\tau (X, Y_0) = (1-\tau)\, d\alpha_0(X, J_0X) + \tau \,
  d\alpha_1(X, J_1X) > 0$.

  \medskip

  We glue now $\Phi_{\mathrm{reg}}$ and $\Phi_{\mathrm{symp}}$ to
  produce a global bundle isomorphism.  Choose a smooth function
  $\rho\colon \Sigma \to [0,1]$ with support in $\Ureg$ such that
  $1-\rho$ has support in $\Usymp$ so that $\rho$ and $1-\rho$ form a
  partition of unity subordinated to $\bigl\{\Ureg, \Usymp\bigr\}$.
  Define a bundle homomorphism
  \begin{equation*}
    \Phi\colon \restricted{TM_0}{\Sigma} \to \restricted{TM_1}{\Sigma}
  \end{equation*}
  by mapping a vector $v\in T_pM_0$ at a point $p\in \Sigma$ to
  $\Phi(v) = \rho(p)\cdot \Phi_{\mathrm{reg}}(v) + (1-\rho(p))\cdot
  \Phi_{\mathrm{symp}}(v)$.  It is obvious that $\Phi$ is a bundle
  homomorphism such that $\restricted{\Phi}{T\Sigma} = \id_{T\Sigma}$
  and such that $\alpha_0 = \alpha_1\circ \Phi$ on
  $\restricted{TM_0}{\Sigma}$ proving properties~(i) and (ii) in the
  lemma.

  It remains to verify property~(iii).  Define the interpolation
  $d\alpha_\tau := (1-\tau)\,d\alpha_0 + \tau\, (d\alpha_1\circ \Phi)$
  for $\tau\in [0,1]$.  Since $\Phi$ agrees with
  $\Phi_{\mathrm{symp}}$ at the points where $\beta = 0$, we obtain
  that $d\alpha_\tau = d\alpha_0$ is non-degenerate at any such point.
  We study now the desired property at points at which $\beta \ne 0$
  and thus $X\ne 0$.
  
  Since $\restricted{d\alpha_\tau}{T\Sigma} = d\beta$ is independent
  of $\tau$, we see that $d\alpha_\tau(X,\cdot)$ vanishes on every
  vector that lies in $\ker\beta$.  Using the same basis chosen above
  with $Y_0 = J_0 X$ and $Y_1 = J_1 X$, it follows that the sign of
  $d\alpha_\tau^n(v_1,\dotsc,v_{2n-2}, X, Y_0) = n\, d\alpha_\tau(X,
  Y_0) \cdot d\beta^{n-1}(v_1,\dotsc,v_{2n-2})$ only depends on the
  sign of the term $d\alpha_\tau(X, Y_0)$.

  For this term we obtain
  $d\alpha_\tau(X,Y_0) = (1-\tau)\,d\alpha_0(X,J_0 X) +
  \tau\,d\alpha_1\bigl(X, \Phi (Y_0)\bigr)$.  The first term is
  clearly positive, and for the second one write
  $d\alpha_1\bigl(X, \Phi (Y_0)\bigr) = \rho \,d\alpha_1(X, J_1X) +
  (1-\rho)\, d\alpha_1\bigl(X, \Phi_{\mathrm{symp}}(Y_0)\bigr)$, where
  the first term is again positive.  Recall that
  $d\alpha_0 = d\alpha_1\circ \Phi_{\mathrm{symp}}$ so that we can
  simplify the second term as
  $d\alpha_1\bigl(X, \Phi_{\mathrm{symp}}(Y_0)\bigr) =
  d\alpha_1\bigl(\Phi_{\mathrm{symp}}(X),
  \Phi_{\mathrm{symp}}(Y_0)\bigr) = \bigl(d\alpha_1\circ
  \Phi_{\mathrm{symp}}\bigr) (X, Y_0) = d\alpha_0(X, Y_0)$.  Thus
  $d\alpha_\tau(X,Y_0) = (1-\tau)\,d\alpha_0(X,Y_0) + \tau\,
  \bigl(\rho \,d\alpha_1(X, Y_1) + (1-\rho)\, d\alpha_0(X, Y_0)\bigr)$
  is positive as a convex combination of positive terms, and we have
  shown property~(iii).
\end{proof}

\begin{proof}[Proof of Proposition~\ref{contact germ determined by
    hypersurface}]
  Let $\alpha_0$ be a contact form for $\xi_0$, and let $\alpha_1$ be
  a contact form for $\xi_1$.  By Lemma~\ref{scaling of defining 1form
    of sing distribution} there is a smooth function
  $f\colon \Sigma \to \RR_{>0}$ such that
  $\iota_0^* \alpha_0 = f\cdot \iota_1^* \alpha_1$.  We would like to
  extend $f\circ \iota_1^{-1}$ to all of $M_1$ to normalize $\alpha_1$
  globally; in general though, if $\Sigma$ is not closed, this might
  be impossible.
  
  Denote the normal bundle of $\iota_1(\Sigma)$ in $M_1$ by
  $\nu_1 \Sigma \stackrel{\pi}{\to} \Sigma$, and recall that there is
  a tubular neighborhood~$U_1$ of $\iota_1(\Sigma)$ that is
  diffeomorphic to a neighborhood $V_1$ of the $0$-section in
  $\nu_1\Sigma$ (of course $V_1$ will generally not have uniform
  radius in the fiber directions, when $\Sigma$ is not closed).  The
  function~$f\circ \pi$ is a smooth positive function on
  $\nu_1\Sigma$.  We will replace $M_1$ by the open subset $U_1$, and
  use $f\circ \pi$ to rescale $\alpha_1$ on $U_1$ so that we can
  assume that $\iota_0^* \alpha_0 = \iota_1^* \alpha_1$.  This allows
  us to apply Lemma~\ref{linear isomorphism around hypersurface} to
  obtain a bundle isomorphism~$\Phi$ between
  $\restricted{TM_0}{\iota_0(\Sigma)}$ and
  $\restricted{TM_1}{\iota_1(\Sigma)}$.

  \medskip

  Let $U_0$ be a tubular neighborhood of $\iota_0(\Sigma)$ in $M_0$
  such that the exponential map~$\exp_0$ (with respect to some
  Riemannian metric) defines a diffeomorphism
  $\exp_0\colon V_0 \to U_0$, where $V_0$ is a neighborhood of the
  $0$-section of the normal bundle of $\Sigma$ in $M_0$.  Similarly,
  let $\exp_1$ be the exponential map on $M_1$.  By suitably reducing
  the size of $U_0$ and $U_1$, we can assume that
  \begin{equation*}
    \Psi := \exp_1 \circ\, \Phi \circ\, \exp_0^{-1}\colon U_0 \to U_1
  \end{equation*}
  is a diffeomorphism.  To simplify our setup, pull-back $\alpha_1$ to
  $U_0$, and work in the fixed ambient manifold~$U_0$.  For simplicity
  we also write $\alpha_1$ for its pull-back.  Then it follows that
  $U_0$ contains the submanifold~$\Sigma$, and carries two contact
  structures given by contact forms~$\alpha_0$ and $\alpha_1$ such
  that $\alpha_0$ and $\alpha_1$ agree at all points of $\Sigma$, and
  such that the linear interpolation of $d\alpha_0$ and $d\alpha_1$ is
  a path of symplectic structures on
  $\restricted{\xi_0}{\Sigma} = \restricted{\xi_1}{\Sigma}$.

  The rest of the proof is a simple application of the Moser trick:
  Clearly the interpolation
  $\alpha_\tau := (1-\tau)\,\alpha_0 + \tau\, \alpha_1$ satisfies
  along $\Sigma$ for every $\tau \in [0,1]$ the contact condition.
  There is thus a small neighborhood of $\Sigma$ in $U_0$ on which all
  $\alpha_\tau$ are contact forms.
  
  As in the standard proof of Gray stability, we define a vector field
  $X_\tau$ on this neighborhood by the equations
  \begin{equation*}
    \alpha_\tau(X_\tau) = 0 \quad \text{ and } \quad
    d\alpha_\tau(X_\tau, \cdot) = f_\tau \, \alpha_\tau - \dot \alpha_\tau
  \end{equation*}
  with $f_\tau := \dot \alpha_\tau (R_\tau)$, where $R_\tau$ is the
  Reeb field of $\alpha_\tau$.  Note that the right hand side of the
  second equation vanishes along $\Sigma$, thus it follows that
  $X_\tau(p) = 0$ at every $p\in \Sigma$.  By reducing to a smaller
  neighborhood of $\Sigma$ in $U_0$, the flow of $X_\tau$ will be
  defined up to time~$1$ giving a contact isotopy between $\xi_0$ and
  $\xi_1$.

  Composing this isotopy with $\Psi$, we find the desired
  contactomorphism.
\end{proof}

\providecommand{\bysame}{\leavevmode\hbox to3em{\hrulefill}\thinspace}
\providecommand{\MR}{\relax\ifhmode\unskip\space\fi MR }
\providecommand{\MRhref}[2]{%
  \href{http://www.ams.org/mathscinet-getitem?mr=#1}{#2}
}
\providecommand{\href}[2]{#2}

\bibliographystyle{amsalpha}


\end{document}